\documentclass[11pt]{amsart}
\usepackage{fancyhdr}
\usepackage{txfonts}

\usepackage{amsmath,amsthm,amssymb,hyperref, mdframed}
\usepackage{tikz, tikz-cd}
\usepackage[all]{xy}
\usetikzlibrary{decorations.markings}
\usetikzlibrary{backgrounds,shapes}
\usepackage{subfig}

\usepackage{color}

\definecolor{aliceblue}{rgb}{0.9, 0.95, 1.0}

\newcommand{\C}{{\mathbb C}}

\newcommand{\pslc}{{\mathrm{PSL}_2 (\mathbb{C})}}
\newcommand{\slc}{{\mathrm{SL}_2 (\mathbb{C})}}
\newcommand{\pslr}{{\mathrm{PSL}_2 (\mathbb{R})}}

\newcommand\HHH{{\mathbb H}}
\newcommand\HH{{\mathbb H}}

\newcommand{\cp}{\mathbb{C}\mathrm{P}^1}
\newcommand\Te{Teichm\"{u}ller }
\newcommand\T{{\mathcal T}}

\newtheorem{theorem}{Theorem}[section]
\newtheorem{prop}[theorem]{Proposition}

\newtheorem{cor}[theorem]{Corollary}

\newtheorem{thm}{Theorem}[section]
\newtheorem*{jthm}{Theorem}

\newtheorem{lem}[thm]{Lemma}
\theoremstyle{definition}
\newtheorem{defn}[thm]{Definition}

\title[Monodromy groups of $\cp$-structures on punctured surfaces]{Monodromy groups of $\cp$-structures\\ on punctured surfaces}

\author{Subhojoy Gupta}

\address{Department of Mathematics, Indian Institute of Science,
Bangalore, India}
\email{subhojoy@iisc.ac.in}

\begin{document}

\begin{abstract} For a punctured surface $S$, we characterize the representations of its fundamental group into $\pslc$ that arise as the monodromy of a meromorphic projective structure on $S$  with poles of order at most two and no apparent singularities.  This proves the analogue of a theorem of Gallo-Kapovich-Marden concerning $\cp$-structures on closed surfaces, and settles a long-standing question about characterizing monodromy groups for the Schwarzian equation on punctured spheres. The proof involves a geometric interpretation of the Fock-Goncharov coordinates of the moduli space of framed $\pslc$-representations, following ideas of Thurston and some recent results of Allegretti-Bridgeland.
\end{abstract}

\maketitle

\section{Introduction}

Let $S_{g,k}$ be an oriented surface of genus $g\geq 0$ and $k \geq 1$ punctures, that has negative Euler characteristic, and let $\Pi$ denote its fundamental group. Let $\mathcal{P}^\circ_g(k)$ be the space of meromorphic projective structures on a surface of genus $g$ and $k$ labelled punctures,  such that each puncture corresponds to a pole of order at most two, and there are no apparent singularities. 

Then, we have a monodromy (or holonomy) map
\begin{equation}\label{mmap}
\Phi : \mathcal{P}^\circ_g({k}) \to  \large\chi(\Pi) 
\end{equation}
where $\chi(\Pi)$ is the \textit{moduli space of $\pslc$-representations} of the surface-group $\Pi$, the space of representations up to conjugation (see \S2.4).  The requirement of no apparent singularities is equivalent to the condition that no puncture has trivial (identity) monodromy around it. We provide an expository account of all these notions in \S2.\vspace{.15mm}

Our main result is:

\begin{thm}\label{thm1} A representation $\rho\mathbin{:}\Pi \to \pslc$ is the monodromy of a meromorphic projective structure in $\mathcal{P}^\circ_g({k}) $ if and only if $\rho$ is non-degenerate, and has non-trivial monodromy around each puncture. 
\end{thm}

\smallskip

Here, a representation is \textit{degenerate} if there is a set $F$ comprising exactly one or two points in $\cp$ such that the image group $\rho(\Pi)$ preserves $F$, and the monodromy around each puncture fixes $F$ (and is parabolic or identity if $F$ is a single point); the representation is \textit{non-degenerate} otherwise -- see Definition \ref{bel}. This notion is related, but distinct from the usual notions of ``non-elementary" or ``irreducible" -- see \S2.4 for examples and a discussion.

\medskip

The question of characterizing such monodromy representations arose in the work of Poincar\'{e} in the context of solving  linear differential equations on a punctured sphere. See \S2.1 and \S2.3 for a brief account of the relation of the Schwarzian equation with the meromorphic projective structures we consider. In this case, by comparing the dimensions of the spaces appearing in \eqref{mmap}, on page 218 of \cite{PActa} Poincar\'{e} writes: {``On peut \textit{en g\'{e}n\'{e}ral} trouver une \'{e}quation du $2^d$ ordre, sans points \`{a} apparence singuli\`{e}re qui admette un groupe donn\'{e}."} Our theorem above thus clarifies his remark, and specifies which monodromy groups do appear, not only for punctured spheres but for punctured surfaces of higher genera as well. \vspace{.15mm}

In more recent times, projective structures in $\mathcal{P}^\circ_g({k})$ have been considered by Luo in  \cite{Luo} where he analyzed the derivative of the map $\Phi$ in \eqref{mmap} when the monodromy around each puncture is non-parabolic (\textit{c.f.} the discussion in \S2.3). In that paper, the holomorphic quadratic differentials having the condition that the order of the pole at each puncture is at most two are called \textit{quasi-bounded}. 
\smallskip

For a closed surface $S_g$ where $g\geq 2$, we can define the representation variety $\chi_g$ comprising representations of $\pi_1(S_g)$ into $\pslc$ upto conjugation, and the space  $\mathcal{P}_g$ of marked $\cp$-structures on $S_g$. We then have a monodromy map
\begin{equation}\label{mmap2}
\Phi_g : \mathcal{P}_g \to  \chi_g
\end{equation}
which is a local homeomorphism by a result of Hejhal in \cite{Hejhal}, and with infinite fibers for each point in its image (this is implicit in \cite{gkm}, see also \cite{Baba1}). 
In this setting, Gallo-Kapovich-Marden proved:

\begin{jthm}[\cite{gkm}]  The image of the monodromy map $\Phi_g$ in \eqref{mmap2} is precisely the set of  non-elementary representations that admit a lift to $\slc$.
\end{jthm}

Their methods are specific to a closed surface; a crucial step in their proof is to obtain, given a non-elementary representation $\rho$, a pants decomposition of the surface such that the restriction of $\rho$ to each pair of pants is a Schottky group. They had raised the question of extending their methods to the punctured case  in \S12.2 of \cite{gkm} (see Problem 12.2.1), which Theorem \ref{thm1} addresses.\vspace{.15mm}

The fact that in the punctured case,  the image is not just the space of non-elementary or irreducible representations, was well-known. Indeed, spherical cone-metrics on a surface are special examples of meromorphic projective structures that have poles of order at most  two at the cone-points; monodromy groups for such structures lie in $\text{PSU}(2)\cong \text{SO}(3)$.  Interestingly, there are examples of such structures with monodromy groups that are degenerate in the sense of Definition \ref{bel} (called \textit{co-axial} in this case --  see \cite{Chenetal} and \cite{Eremenko}), however they would necessarily have apparent singularities (i.e.\ cone-angles that are integer multiples of $2\pi$), and hence will not be $\cp$-structures in $\mathcal{P}^\circ_g({k})$. 

\medskip

\textit{About the proof.} Our proof of Theorem \ref{thm1}  utilizes the notion of \textit{framed} representations and their moduli space $\widehat{\chi}(\Pi)$, due to Fock-Goncharov (\cite{FG}), and the framed monodromy map 
\begin{equation}\label{mmap3}
\widehat{\Phi}: \mathcal{P}^\ast_g({k}) \to \widehat{\chi}(\Pi)
\end{equation}
 due to Allegretti-Bridgeland (\cite{AllBrid}), where $\mathcal{P}^\ast_g({k})$ is the space of \textit{signed} projective structures (defined in \S2.5). 
 Indeed, Theorem \ref{thm1} follows from

\begin{thm}\label{thm2} 
The image of the framed monodromy map $\widehat{\Phi}$ in \eqref{mmap3} is precisely the set of non-degenerate framed representations that have non-trivial monodromy around each puncture. 
\end{thm}

Here, the notion of \textit{non-degenerate framed representations} was introduced by Allegretti-Bridgeland in \cite{AllBrid} (see Definition \ref{degen} and Theorem \ref{ab-thm61}). The relation with Definition \ref{bel} -- that accounts for the similarity in nomenclature --  is clarified in Proposition \ref{prop1}.
They showed that a non-degenerate framed representation admits a signed triangulation $T$ on $S_{g,k}$ with non-zero complex cross-ratios for each edge (also known as Fock-Goncharov coordinates for the framed representation) -- see Theorem \ref{ab-thm91}. This can be thought of as the analogue of finding a Schottky pants-decomposition in the work of Gallo-Kapovich-Marden, and is a crucial ingredient in our proof. 

The key idea of our proof is to then interpret the complex cross-ratios geometrically to determine a pleated plane in hyperbolic $3$-space. This in turn determines a $\cp$-structure on the punctured surface via the operation of grafting the ``straightened" plane, following ideas of Thurston (see \S2.2).  Finally, we analyze the Schwarzian derivative of the developing map of the resulting structure around the punctures to conclude that we have obtained our desired structure in $\mathcal{P}^\ast_g({k})$.

A similar strategy had been employed in \cite{GM2} in the case of poles of order \textit{greater} than two. There, the analysis around the punctures was different (and easier); in particular, the ``straightening" of the pleated plane is always the developing map of a crowned hyperbolic surface. As we shall see in \S3.2 and \S3.3, in the present case we obtain either a cusp or geodesic boundary in the quotient surface, and we need to consider these two possibilities separately. 

\smallskip 

Our argument, together with the operation of ``$2\pi$-grafting" (used, for example, in \cite{Goldman}) yields the following immediate consequence:

\begin{thm}[Infinite fibers] \label{thm3} For any representation $\rho$ in the image of $\Phi$ in \eqref{mmap}, the preimage $\Phi^{-1}(\rho)$ is an infinite subset of $ \mathcal{P}^\circ_g({k})$. 
\end{thm}

Note that it follows from the main result of \cite{Luo} that such a fiber $\Phi^{-1}(\rho)$ is a discrete subset of $ \mathcal{P}^\circ_g({k})$, at least when $\rho$ lies in the smooth part of  $\large\chi(\Pi)$ and has non-parabolic monodromy around the punctures. 

\smallskip

To conclude the introduction, we note that the work of Allegretti-Bridgeland in \cite{AllBrid} considered meromorphic projective structures which could have poles of orders greater than two, \textit{as well as} poles of orders at most two, at the punctures.  For a $k$-tuple of integers $\mathfrak{n} = (n_1,n_2,\ldots n_k)$, let $\mathcal{P}_g^\ast(\mathfrak{n})$ be the space of signed $\cp$-structures on $S_{g,k}$ with the $i$-th puncture having a pole of order $n_i\geq 0$. This admits a monodromy map $F$ to the corresponding space of framed representations  $\widehat{\chi}_g(\mathfrak{n})$. By combining the techniques in \cite{GM2} and this paper, we deduce:

\begin{thm}\label{thm4}
The image of the monodromy map $F: \mathcal{P}_g^\ast(\mathfrak{n}) \to \widehat{\chi}_g(\mathfrak{n})$ is precisely the subset of non-degenerate framed representations  that have non-trivial monodromy around the punctures corresponding to the poles of order at most two.
\end{thm}

This answers a question of Allegretti-Bridgeland (see \S1.7.1 of \cite{AllBrid}) in full generality.\\

\smallskip

\textbf{Acknowledgments.} I acknowledge the SERB, DST (Grant no. MT/2017/ 000706), the UGC Center for Advanced Studies grant, and the Infosys Foundation for their support. This paper arose out of discussions with Mahan Mj regarding related projects, and I am grateful to him for his help and encouragement. I would also like to thank Dylan Allegretti, Shinpei Baba, Indranil Biswas and Sorin Dumitrescu for several interesting conversations. Finally, I would like to thank the anonymous referees for their comments that helped improve article.

\section{Preliminaries}

\subsection{$\cp$-structures on a surface} 

A \textit{$\cp$-structure} or \textit{complex projective structure} on a (possibly open) surface $S$ is a maximal atlas of charts to $\cp$ such that the transition maps are restrictions of M\"{o}bius maps, i.e.\ elements of $\text{Aut}(\cp) = \pslc$. It is thus a $(G,X)$-structure where $G = \pslc$ and $X = \cp$, and can be equivalently described as a pair of a developing map $f:\widetilde{S} \to \cp$ (defined on the universal cover) and a holonomy or monodromy representation $\rho: \pi_1(S) \to \pslc$, where $f$ is a $\rho$-equivariant immersion.

A map between a pair of $\cp$-structures is a \textit{projective isomorphism} if it is a homeomorphism that is locally a restriction of a M\"{o}bius map. For a closed  surface  $S_{g}$ where $g\geq 0$, we can define the space  $\mathcal{P}_{g}$ of marked $\cp$-structures on $S_g$, up to projective isomorphisms preserving the marking. Here, a \textit{marking} is a choice of a homeomorphism from a fixed $S_g$ to our projective surface, up to homotopy. This space admits a forgetful map $\pi: \mathcal{P}_{g} \to \T_g$, where $\T_g$ is the \Te space of  $S_g$, the space of marked Riemann surfaces up to conformal homeomorphism preserving the marking. See \cite{Dum} for an expository account.\vspace{.15mm}

The connection with differential equations and complex analysis arises from the fact that each fiber $\pi^{-1}(X)$ where $X \in \T_g$ can be identified with the vector space $Q(X)$ of holomorphic quadratic differentials on the Riemann surface $X$. We briefly sketch the proof:

In one direction, given a holomorphic quadratic differential $q\in Q(X)$, we obtain a $\cp$-structure by passing to the universal cover $\widetilde{X}$, and solving the \textit{Schwarzian equation} 
\begin{equation}\label{schwarz}
u^{\prime\prime}(z) + \frac{1}{2} \tilde{q} u(z) = 0.
\end{equation} 
Two linearly independent solutions $u_0,u_1$ of his second-order linear differential equation then define the developing map $f:= {u_0}/{u_1}$ which is $\rho$-equivariant for a monodromy representation $\rho:\pi_1(S_g) \to \pslc$. 

In the other direction, if we fix a reference $\cp$-structure $X_0$ (e.g.\ the uniformizing structure), then for any other $\cp$-structure on $X$, the \textit{Schwarzian derivative} of the developing map
\begin{equation}\label{schwarz-deriv}
S(f) = \left(\left(\frac{f^{\prime\prime}}{f^\prime}\right)^\prime - \frac{1}{2}\left(\frac{f^{\prime\prime}}{f^\prime}\right)^2\right) dz^2
\end{equation}
is a holomorphic quadratic differential on $\widetilde{X}$ that descends to one on $X$, defining an element $q\in Q(X)$. Moreover, these directions are inverses of each other; for details, see for example \cite{Gunn} or \cite{Hubbard}.

\subsection{Grafting} Grafting is a geometric operation that results in a $\cp$-structure by deforming a hyperbolic structure, that goes back to Maskit (\cite{Maskit}),  Hejhal (\cite{Hejhal}), and Sullivan-Thurston (\cite{ThuSull}), and was developed in unpublished work of Thurston. \vspace{.15mm}

The following definitions would be useful in our sketch of the construction:

\begin{defn}[Embedded lune]  A \textit{lune} $L_\alpha$ of angle $\alpha \in (0, 2\pi]$ is the region of $\cp$ bounded by two semi-great-circles $\partial_{\pm} L_\alpha$ with common  endpoints $0,\infty$, with an angle $\alpha$ between them.  In the affine chart $\C$, one can choose $\partial_-L_\alpha$ to be the positive imaginary axis, and $\partial_+L_\alpha$ to be the ray at an angle $\alpha$ from the positive imaginary axis.  Note that $L_{2\pi}$ can thus be thought of as $\cp$ slit along the positive imaginary axis between $0$ and $\infty$.
\end{defn}

\textit{Remark.}  The region $L_\alpha \subset \cp$ defined above acquires a standard $\cp$-structure, and by a ``lune" we shall often mean a surface equipped with a  $\cp$-structure that is \textit{isomorphic} to such a region. \vspace{.15mm}

\begin{defn}[Immersed lune]\label{imlune} In the case the angle $\alpha> 2\pi$, a lune $L_\alpha$ is defined as above, with the excess angle passing over to another copy of $\cp$.  In other words, $L_\alpha$ is a conformal surface is obtained by taking $\lfloor \frac{\alpha}{2\pi}\rfloor$ additional copies of $\cp$, each slit along the positive imaginary axis, identified as chain (with the right side of the slit identified with the left side of the slit in the next copy of $\cp$ and so on), such that the last copy of $\cp$ has an embedded lune of the remaining angle $\alpha - 2\pi \cdot \lfloor \frac{\alpha}{2\pi}\rfloor$. 
This lune $L_\alpha$ for $\alpha>2\pi$  admits an immersion to $\cp$ that projects each copy of $\cp$ to a fixed $\cp$.

\end{defn}

 In the simplest case, the  grafting construction is as follows:  Recall that a hyperbolic surface $X$ is a $\cp$-structure with a developing image that is a round disk $\Omega \subset \cp$ that is invariant under a Fuchsian group $\Gamma < \pslr$ where $\Gamma$ is the image of the monodromy representation for $X$. For  a simple closed separating geodesic $\gamma$ on  $X$ with weight $\alpha>0$, we then form the new domain $\Omega^\prime$ immersed in $\cp$ by inserting 
a lune $L_\alpha$ at the developing image of each lift of $\gamma$ in $\Omega$. This new (possibly immersed) domain acquires an action  by a M\"{o}bius group $\Gamma_\alpha < \pslc$ (though the quotient maybe a non-Hausdorff space, e.g.\ when $\Gamma_\alpha$ is dense in $\pslc$). Here, $\Gamma_\alpha$ is the unique representation that restricts to $\Gamma$ on one of the components of $X \setminus \gamma$, and on the other, restricts to the conjugate of $\Gamma$ by an elliptic element that rotates around the axis of $\gamma$ by an angle $\alpha$.  The pair $(\Omega_\alpha, \Gamma_\alpha)$ then defines the new projective surface $X_\alpha$ obtained by grafting $X$ along the weighted simple closed curve $\alpha \cdot \gamma$.

 A  \textit{measured lamination} $\lambda$ is a closed subset of $X$ that is a disjoint collection of complete geodesics, equipped with a transverse measure. (See, for example, \cite{thurstonnotes}.) The construction above extends to grafting $X$ along $\lambda$ by a limiting argument. 
 
 \vspace{.15mm}
 
 In fact, Thurston showed that on a closed surface $S_g$ of genus $g\geq 2$, \textit{any} $\cp$-structure is obtained via this grafting construction; see \cite{KamTan} and \cite{Tanig}, and the more recent exposition in \cite{BabaExp}. Although we shall not need this result, we do need a key ingredient of Thurston's proof, the notion of a \textit{pleated plane} defined by the grafting construction as follows:

To define this, consider first an isometric $\Gamma$-invariant embedding $\Psi_0:\widetilde{X} \to \mathbb{H}^3$ with image on the totally-geodesic equatorial plane in $\mathbb{H}^3$. In particular, the $\Psi_0$-image of the leaves of the lifted lamination $\widetilde{\lambda}$ on the universal cover $\widetilde{X}$, defines a collection of disjoint geodesic lines in $\mathbb{H}^3$.
The pleated plane $\Psi_\alpha: \widetilde{X} \to \mathbb{H}^3$ is then a $\Gamma_\alpha$-equivariant immersion obtained by bending one side of each such geodesic line with respect to the other, by an angle equal to the weight (i.e.\  transverse measure) on that leaf.  (See Figure 1.)  As before, this construction can be defined for any measured lamination $\lambda$ by a limiting argument. Moreover, the initial hyperbolic surface $X$ need not be compact for this construction; indeed,  one can define grafting the Poincar\'{e} disk $\mathbb{D}$ along a collection $\tilde{\lambda}$ of disjoint geodesic line to obtain a simply-connected $\cp$-structure (see Figure 1, and  Theorem 10.6 of \cite{Kulkarni-Pinkall} for the analogue of Thurston's result in this setting).  

\begin{figure}
  \centering
  \includegraphics[scale=0.35]{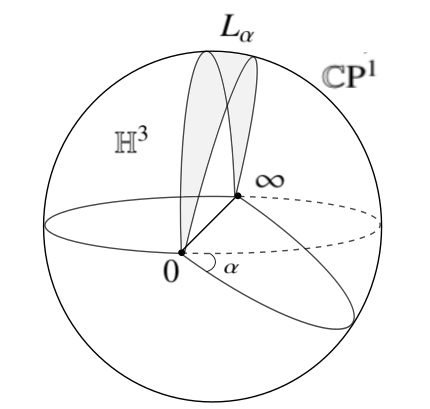}
  \caption{A lune $L_\alpha$ on the ideal boundary and the corresponding pleated plane, in the case $\mathbb{D}$ is grafted along a single geodesic line of weight $\alpha$. }
\end{figure}

Finally, we note that in the above description of grafting,  adding a lune $L_{2\pi n}$ of an angle that is an integer-multiple of $2\pi$ (i.e.\ $n\in \mathbb{N}$) to an already-existing lune in the developing image results in a new projective structure (that wraps more times around $\cp$) but has identical holonomy. Indeed, a $2\pi$-rotation around a pleating line does not affect the  pleated plane  defined by the grafting construction.  This is called $2\pi$-\textit{grafting} and has been used as an effective tool to study fibers of the monodromy map \eqref{mmap2} in the case of a closed surface (see, for example, \cite{Goldman},  \cite{Baba2} and \cite{Baba3}). We shall use $2\pi$-grafting in the proof of Theorem \ref{thm3}.

\subsection{Meromorphic projective structures}
Consider a punctured surface  $S_{g,k}$ (where $k\geq 1$) of negative Euler characteristic. Fix a choice of a reference $\cp$-structure $X_0$  on the closed surface $S_g$ (obtained by filling in the punctures).

We shall restrict to those $\cp$-structures on $S_{g,k}$ that satisfy the following two requirements:
\begin{itemize}
\item[(i)] with respect to the choice $X_0$, the developing map has a Schwarzian derivative with a pole of order less than or equal to two at the punctures, and
\item[(ii)] the monodromy around each puncture is not the identity element in $\pslc$. 
\end{itemize}
Requirement (ii) is the property of having \textit{no apparent singularities}. 
This is precisely the definition used in  \cite{AllBrid} (and elsewhere in the literature -- see, for example, \cite{Iwasaki}): an \textit{apparent singularity} of the Schwarzian equation is a ``regular singularity" around which the (projectivized) monodromy is trivial.  An example of an apparent singularity is a ``branch-point" of a projective structure -- see \cite{Mandelbaum}. 
In what follows we shall briefly discuss how this matches with the classical notion that, as mentioned in the Introduction,  historically arose in the study of linear differential equations.

First, recall that a \textit{regular singularity} of the Schwarzian equation \eqref{schwarz} is a pole of the coefficient function $\tilde{q}$ that has order at most two (\textit{c.f.} requirement (i) above).  This is part of the usual ``Fuchsian" theory of linear differential equations -- see for example, \cite{Hille}.
The idea is that in a neighborhood of a regular singularity $a$, a solution $u(z)$ of \eqref{schwarz} would have a single  (convergent) power series expression; note that this solution need not be meromorphic (i.e.\ with a finite order pole at $a$) since it could involve a logarithmic term (as in $u_1$ in possibility (b) below). 
In contrast, at an \textit{irregular} singularity, a solution would have a power series expression that is only convergent in sectors, and have different asymptotic behaviour in different sectors -- see \cite{Wasow} or \cite{Sib-book} for a more detailed analysis.


From a formal power series argument, it follows that if the regular singularity is at $a=0$, and the coefficient in \eqref{schwarz} is of the form 
\begin{equation}\label{order2-exp} 
\tilde{q}  = \frac{1- \theta^2}{2z^2} + \cdots
\end{equation}  then there are the following two possibilities (for details see, for example, Theorem IX.1.1. of \cite{Saint-Gervais}, or \cite{Iwasaki-book} or \cite{Eremenko2}, or section 5.2 of \cite{AllBrid}): 
\begin{itemize}
\item[(a)]  $\theta \notin \mathbb{Z}$, in which case there are two linearly independent solutions of the form $y^{\pm \theta}$ in a neighborhood $U$ of $0$, where $y(z)$ is a local coordinate on $U$, and the monodromy in $\pslc$ around the regular singularity is multiplication by $e^{2\pi i \theta}$.

\item[(b)] $\theta  \in \mathbb{Z}$, in which case the two solutions are of the form $u_0(y) = y^{\lvert \theta \rvert}$ and $u_1(y) = y^{-\lvert \theta \rvert} + C \ln y$ for some local coordinate $y$ around $0$. If the constant $C=0$, then the monodromy around $0$  is trivial; if the constant $C \neq 0$ then the monodromy around $0$ is a non-trivial parabolic element.
\end{itemize}

The case when $C=0$ in possibility (b) above is an alternative definition of an apparent singularity that often occurs in the literature (see, for example \cite{Iwasaki-book}).  Thus, yet another equivalent definition is: an apparent singularity is a regular singularity around which any solution of the Schwarzian equation is meromorphic (see, for example, \cite{Ince} or \cite{Iwasaki-book}).  In Poincar\'{e}'s definition of an  ``apparent singularity" in  page 217 of \cite{PActa}, it is clear that the (projectivized) monodromy around it is trivial, hence it matches with the preceding discussion.  See Proposition IX.1.2 of \cite{Saint-Gervais} for a summary of these equivalent definitions. 

We note that the definition in \cite{Luo} is slightly different: after  making corrections to be consistent with the standard theory (see also Lemma \ref{pmon} below), his definition of an apparent singularity (on the first page of his paper) simply requires that $\theta \in \mathbb{Z}$. 
In this paper, in contrast, a non-apparent singularity is allowed to have parabolic monodromy around it. \\

Now, let $\mathcal{P}_g(k)$ denote the space of such marked $\cp$-structures on $S_{g,n}$ satisfying only the requirement (i) stated at the beginning of this subsection, up to projective isomorphisms that preserve the marking (this includes a labelling of the punctures). There is a forgetful map  $\hat{\pi}: \mathcal{P}_g(k) \to \T_{g,k}$, and via the Schwarzian derivative, a fiber $\hat{\pi}^{-1}(X)$ can be identified with the vector space $Q_g(k)$ of meromorphic quadratic differentials with poles of order at most  two, at the punctures. 

At each puncture, a quadratic differential $q\in Q_g(k)$ has the form 
\begin{equation}\label{qexp}
q = \left(\frac{\alpha}{z^{2}} + \frac{\beta}{ z}  +  g(z) \right)dz^2
\end{equation} 
in local coordinates around the puncture, where $\alpha, \beta \in \mathbb{C}$ and $g(z)$ is holomorphic.

Following  \cite{AllBrid}, we define the \textit{leading coefficient} of $q$ at the  pole to be $\alpha$ in \eqref{qexp},  the  \textit{residue} of $q$ at the pole to be  $\pm 4\pi i\sqrt{\alpha}$, and the \textit{exponent} to be $\pm 2\pi i \sqrt{1+ 4\alpha}$. It is easy to check that the  leading coefficient is independent of the choice of local coordinates around the puncture, and note that  the residue and the exponent are well-defined up to sign. 

It was shown in Lemma 5.1 of \cite{AllBrid} that:

\begin{lem}\label{pmon}  If the Schwarzian equation is written as $u^{\prime\prime}(z) - q \cdot u(z) = 0$, the monodromy in $\slc$ around the puncture has eigenvalues equal to $-e^{\pm r/2}$ where $r$ is the exponent, as defined above.
\end{lem}

\textit{Remark.} When the Schwarzian equation  is written as \eqref{schwarz}, this matches with the discussion earlier in the section when the coefficient $\alpha = (\theta^2 -1)/4$ (see \eqref{order2-exp}).  In particular, if $\theta \in \mathbb{Z}$, then the exponent $r$ is an integer multiple of $2\pi i$ (up to a sign), and the (projectivized) monodromy is trivial or parabolic, as mentioned in possibility (b). Moreover, if $\theta =0$, then the singularity is ``logarithmic" and the monodromy is parabolic (see page 238 of \cite{Saint-Gervais}, and Example 3.4 of \cite{AllBrid}).\\

The subset $\mathcal{P}^\circ_g(k) \subset \mathcal{P}_g(k)$ comprises those $\cp$-structures with no apparent singularities, i.e.\ satisfying requirement (ii) above. In \cite{AllBrid}, Allegretti-Bridgeland showed that this is a dense, open subset.
By the Remark above, the exponent of the Schwarzian derivative $q$ at each puncture cannot be zero; this shall later enable us to define a ``sign" at each puncture.
 
 \smallskip
 
As mentioned at the end of \S1,  \cite{AllBrid} considered poles of order greater than two at the punctures as well, and they introduced a more general notion of meromorphic projective structures that could have poles of higher order (see Definition 3.1 in their paper). We deal with the case that all poles are of order greater than two in \cite{GM1, GM2}.  

\subsection{Surface-group representations} 

The \textit{moduli space of $\pslc$-representations} of a surface-group $\Pi$ 
\begin{equation}\label{chip} 
\chi(\Pi) = \text{Hom}(\Pi, \pslc)/\pslc
\end{equation}
is the space of representations $\Pi$ into $\pslc$ up to conjugation, i.e.\ $\rho_1\sim\rho_2$ if and only if $ \rho_2 = g \cdot \rho_1  \cdot g^{-1}$ for some $g\in \pslc$. This is a non-Hausdorff space; to rectify this, usually one considers the quotient  in the following extended sense: $\rho_1\sim\rho_2$ if and only if the closures of their $\pslc$-orbits intersect.  
The latter quotient  coincides with the geometric invariant theory quotient, and yields the \textit{$\pslc$-character variety} -- see \cite{Dolgachev} or \cite{Newstead};  the equivalence classes  are in bijection with \textit{$\pslc$-characters} via the map  $$[\rho] \mapsto \chi_\rho\mathbin{:}\Pi \to \C \text{ defined by } \chi_\rho(\gamma) = \text{tr}^2\rho(\gamma) \text{ for each } \gamma \in \Pi, $$see \cite{Heusener-Porti} for details. \vspace{.15mm}

However, in this article we shall stick with the quotient space \eqref{chip} since the notion of ``degenerate" representations (defined next) is a  condition on conjugacy classes, and not on characters (see Example 1 below, and the remark before that). 
\smallskip

Throughout, we shall identify  $\cp$ as $\partial_\infty \mathbb{H}^3$, the ideal boundary of hyperbolic $3$-space; recall that the action by an element of a M\"{o}bius group on $\cp$ extends to an isometry of $\mathbb{H}^3$. 
Recall that a representation $\rho$ (or its equivalence class in the character variety) is said to be
\begin{itemize}
\item[(i)] \textit{elementary} if either there is a point $p \in \mathbb{H}^3 \cup \cp$, or a pair of points $\{p_1,p_+\} \subset \mathbb{H}^3 \cup \cp $, that are preserved by $\rho(\gamma)$ for each $\gamma \in \Pi$, 
\item[(ii)] \textit{unitary} if there is an interior point $p \in \mathbb{H}^3$ that is fixed by $\rho(\gamma)$ for each $\gamma \in \Pi$, 
\item[(iii)] \textit{non-elementary} if $\rho$ is not elementary,
\item[(iv)] \textit{reducible} if there is a point of $\cp$ that is fixed by $\rho(\gamma)$, for each $\gamma \in \Pi$,  and
\item[(v)] \textit{irreducible} if there is no point of $\cp$ as in (iv). 
\end{itemize}

\smallskip
A key new definition, inspired by the work in \cite{AllBrid} (see Definition \ref{degen} and Proposition \ref{prop1}), is the following:

\begin{defn}\label{bel} A representation  $\rho \in \text{Hom}(\Pi, \pslc)$ is said to be \textit{degenerate} if and only if one of the following properties is satisfied:
\begin{itemize}
\item[(a)] there is a single point $p_0\in \cp$ such that the monodromy around each puncture is a parabolic element with fixed point $p_0$ or the identity element,  and $\rho(\gamma)$ fixes $p_0$ for each $\gamma \in \Pi$.
\item[(b)] there is a pair of points $F=\{p_-, p_+\} \in \cp$ such that the monodromy around each puncture fixes the points in $F$, and $\rho(\gamma)$ preserves the set $F$, for each $\gamma \in \Pi$.  
\end{itemize}
We say $\rho$ is \textit{non-degenerate} if it is not degenerate.  
\end{defn} 

\textit{Remark.} This property is clearly  invariant if $\rho$ is conjugated by any element of $\pslc$.  It is well-known (see, for example, Lemma 3.15 of \cite{Heusener-Porti}) that if $\rho$ is irreducible, then $\rho$ and $\rho^\prime$ define the same point in the \textit{$\pslc$-character variety} if and only if  $\rho$ and $\rho^\prime$ are conjugate representations. However, the following example shows that for the character of a  \textit{reducible} representation may arise from more than one $\pslc$-orbit, where one orbit comprises non-degenerate representations, and another comprises degenerate ones. 

\medskip

\textit{Example 1.} Let $\alpha$ and $\beta$ be the two generators of $\pi_1(S_{0,3})$ (both loops around punctures), and define a representation $\rho_1$  by 
\begin{center}
$\rho_1(\alpha) = \begin{pmatrix}\lambda  & 0 \\ 0 & \lambda^{-1}\end{pmatrix}$ \hspace{.1in}  and \hspace{.1in}  $\rho_1(\beta) = \begin{pmatrix}\lambda  & t(\lambda - \lambda^{-1}) \\ 0 & \lambda^{-1}\end{pmatrix}$ 
\end{center}
where $\lambda \neq 0,1$ and $t\in \mathbb{C}^\ast$. These are then two loxodromic elements of $\pslc$  having \textit{distinct} axes, but with a common fixed point  $\infty \in \cp$.  It is easy to verify that $\rho_1$ is non-degenerate. However, if $g_n = \begin{pmatrix} 1/n & 0 \\ 0 & n\end{pmatrix}$ then  the conjugation $g_n \rho_1 g_n^{-1} \to \rho_0$ as $n\to \infty$ where $\rho_0$ is the degenerate representation defined by $\rho_0(\alpha) = \rho_0(\beta) =  \begin{pmatrix}\lambda  & 0 \\ 0 & \lambda^{-1}\end{pmatrix}$. Since $\rho_0$ lies in the $\pslc$-orbit closure of $\rho_1$, we have $[\rho_0] = [\rho_1]$.  

\medskip

It is easy to see from the preceding definitions that a degenerate representation is an elementary representation; however  the converse is not true, as the following examples show.

\medskip

\textit{Example 2(i).} Consider a \textit{dihedral representation} $\rho_2$ of $\pi_1(S_{0,3})$,  defined by $\rho_2(\alpha) = \begin{pmatrix}\lambda  & 0 \\ 0 & \lambda^{-1}\end{pmatrix}$ and $\rho_2(\beta) = \begin{pmatrix}0  & -\lambda \\  \lambda^{-1} & 0 \end{pmatrix}$. Note that although $\rho_2(\beta)$ preserves $\{0,\infty\}$, it does not fix them; $\rho_2$ is thus elementary but non-degenerate.

\smallskip

\textit{Example 2(ii).} Consider any representation of $\Pi$ that sends the generators to elliptic elements, such that all axes pass through a common point. Such a representation is unitary, and hence elementary, but if there are at least two distinct axes then it will be non-degenerate.  

\medskip
Similarly,  the notion of degenerate is different from reducibility.  In one direction, note that the representation $\rho_1$ in the example above is reducible (since the image group has a global fixed point in $\cp$), but non-degenerate. In the converse direction, we shall construct a representation $\rho_3\mathbin{:}\pi_1(S_{1,2}) \to \pslc$ that is irreducible, but degenerate in the following example:

\smallskip

\textit{Example 3.} Consider the presentation $\pi_1(S_{1,2}) = \langle u,x,y,a,b \ \vert\ a = uxy,\ b=uyx\ \}$ (see \S5.3 of \cite{Goldman-Trace}), and define the representation $\rho_3$  by $\rho_3(u) = \begin{pmatrix}\lambda  & 0 \\ 0 & \lambda^{-1}\end{pmatrix}$ and  $\rho_3(x) = \rho_3(y) =  \begin{pmatrix}0  & -1 \\ 1 & 0\end{pmatrix}$. Note that the images of the generators preserve the pair $\{0,\infty\}$, and the monodromy elements around the punctures, namely  $\rho_3(a)$ and $\rho_3(b)$, fix this pair. Thus, $\rho_3$ satisfies (b) in Definition \ref{bel}; however $\rho_3$ is irreducible as there is no global fixed point. 

\medskip

For \textit{punctured spheres}, however, irreducible does imply non-degenerate:

\begin{lem} Let $k\geq 3$ and let $\rho\mathbin{:}\pi_1(S_{0,k}) \to \pslc$ be a degenerate representation. Then $\rho$ is reducible.
\end{lem}
\begin{proof} If $(a)$ in Definition \ref{bel} holds, then there is certainly a global fixed point in $\cp$.  If $(b)$ holds, then the monodromy around each puncture fixes the pair $\{p_-,p_+\}$, and since the fundamental group of a punctured sphere is generated by loops around the punctures, the entire monodromy group will fix this pair. 
\end{proof}

\subsection{Framed representations} 

A framed representation $\hat{\rho}$  comprises a representation   $\rho \in \text{Hom}(\Pi, \pslc)$ together with a framing, that is, a choice of a fixed point in $\cp$ (an eigenline in $\C^2$) for the monodromy around each puncture (see \S4.1 of \cite{AllBrid}). Alternatively, as in \cite{FG},  a framing is a $\rho$-equivariant map $\beta_\rho\mathbin{:} F_\infty \to \cp$ where $F_\infty$ or the \textit{Farey set} is the set of points on the ideal boundary that correspond to the lifts of the punctures on $S_{g,k}$. These definitions are equivalent when we consider framed representations up to an action of $\pslc$ (see \eqref{chiph}). For example, the first description assigns a point in $\cp$ to each ideal boundary point of a fundamental domain $F$ in the universal cover of $S_{g,k}$; this extends equivariantly  to define a map $\beta_\rho$. Conversely, restricting the map $\beta_\rho$ to a fundamental domain $F$ assigns a point to each puncture; the equivariance ensures each is a  fixed point of the monodromy around the puncture.  In these constructions, a different choice of a fundamental domain changes $\beta_\rho$ exactly by a post-composition by an element of $\pslc$. 

The following definition is adapted from Definition 4.3 of \cite{AllBrid}: 

\begin{defn}\label{degen} A framed representation $\hat{\rho} = (\rho, \beta_\rho)$ is said to be \textit{degenerate} if either of the following conditions hold:\\
(1) The image of the map $\beta_\rho$ is a single point $p_0 \in \cp$,  the monodromy around each puncture is parabolic with fixed point $p_0$ or the identity element, and $\rho(\gamma)$ fixes $p_0$ for each $\gamma \in \Pi$, or \\
(2) The image of the map $\beta_\rho$ is a pair of points $\{p_-,p_+\} \in \cp$, that is fixed by the monodromy around each puncture, and preserved (i.e.\ fixed or permuted) by $\rho(\gamma)$ for each $\gamma \in \Pi$.
\end{defn}

\textit{Remark.}  In the terminology of \cite{AllBrid}, this is the case of a \textit{closed} surface in their Definition 4.3.  For the general case, there is an additional condition (3) --- see the proof of Theorem \ref{thm4} in \S3.4. 
Our conditions (1), (2) and (3) are equivalent  to (D3), (D2) and (D1) of their definition, respectively. To translate from their terminology to ours, note that (i) a flat section of the bundle arising from the representation, is by definition invariant under parallel transport, and hence preserved by the monodromy, and (ii) choosing a flat section in a neighborhood of a puncture is equivalent to choosing a fixed point of the monodromy around it.

\smallskip

Let  $\widehat{\text{Hom}}({\Pi}, \pslc)$ be the space of framed representations.  We have an action of $\pslc$ on this space defined by  $A\cdot (\rho, \beta_\rho) = (A\rho A^{-1}, A\cdot \beta_\rho)$ for each $A\in \pslc$. 

The \textit{moduli space of framed representations} is 
\begin{equation}\label{chiph}
\widehat{\chi}(\Pi)= \widehat{\text{Hom}}({\Pi}, \pslc)/\pslc
\end{equation} 
where the (possibly non-Hausdorff) quotient is by the action described above.  As for Definition \ref{bel}, the definition of a degenerate \textit{framed} representation as above is well-defined up to conjugation, and hence it makes sense to talk of elements of $\widehat{\chi}(\Pi)$ being degenerate or non-degenerate. The quotient space is $\widehat{\chi}(\Pi)$ in fact an algebraic stack, but we shall ignore this since we shall be interested in the subset $\widehat{\chi}(\Pi)^\ast$ of equivalence classes of \textit{non-degenerate} framed representations that is Zariski open (Lemma 4.5 of \cite{AllBrid}) and a complex manifold (Corollary 9.9 of \cite{AllBrid}). \vspace{.15mm}

We also define the space of \textit{signed} projective structures $\mathcal{P}^\ast_g({k})$ where we choose a sign of the exponent  at each puncture (see \S2.3). This corresponds to a choice of an eigenline in $\mathbb{C}^2$ (or, equivalently, a fixed point in $\cp$) for the monodromy around each puncture. Thus, $\mathcal{P}^\ast_g({k})$ is a branched cover over $\mathcal{P}^\circ_g({k})$, and we have a map $\pi_0: \mathcal{P}^\ast_g({k}) \to \mathcal{P}^\circ_g({k})$ which is generically $2^k$-to-one.

There is a well-defined monodromy map $\widehat{\Phi}: \mathcal{P}^\ast_g({k}) \to  \widehat{\chi}(\Pi)$ that records the usual monodromy $\rho$ on the punctured surface, together with a framing that assigns to each puncture the choice of the eigenline of the monodromy around it, as determined by the sign of the exponent.

Moreover, Allegretti-Bridgeland have proved:

\begin{thm}[Theorem  6.1 of \cite{AllBrid}]\label{ab-thm61} The image of the monodromy map $\widehat{\Phi}$ is contained in the space of non-degenerate framed representations.
\end{thm}

For proving Theorem \ref{thm1}, we would need to consider the representations in the moduli space of representations $\chi (\Pi)$,  that is the image of the forgetful projection $\pi: \widehat{\chi}(\Pi) \to \chi(\Pi)$ defined by $\pi(\rho, \beta_\rho) = \rho$. 

The commutative diagram 
\begin{equation}\label{commd} 
\begin{tikzcd}
\mathcal{P}^\ast_g({k})  \arrow{r}{\widehat{\Phi}} \arrow[swap]{d}{\pi_0} & \widehat{\chi}(\Pi) \arrow{d}{\pi} \\
\mathcal{P}^\circ_g({k})  \arrow{r}{\Phi} & {\chi}(\Pi) 
\end{tikzcd}
\end{equation}
summarizes the relationship between the various spaces.

\subsection{Pleated plane and Fock-Goncharov coordinates}

In what follows let $\widetilde{S}$ denote the universal cover of $S_{g,k}$; it is useful to equip $S_{g,k}$ with a complete hyperbolic metric with finite volume (so that the $k$ punctures are cusps) -- the universal cover $\widetilde{S}$ can then be identified with the Poincar\'{e} disk $\mathbb{D}$.\vspace{.15mm}

Recall that an \textit{ideal triangulation} $T$ of $S_{g,k}$ is a collection of disjoint, homotopically non-trivial arcs with endpoints at the punctures, that divides the surface into triangles. (The terminology arises from the fact that when each arc is homotoped to a geodesic representative in the hyperbolic metric we have fixed, the complementary regions are isometric to ideal triangles in the hyperbolic plane.) 
The ideal triangulation $T$ lifts to an ideal triangulation $\widetilde{T}$ of $\widetilde{S}$, equipped with a  $\Pi$-action.\vspace{.15mm}

We shall consider a  \textit{pleated plane} determined by a map $\Psi: \widetilde{S} \to \HHH^3$ such that each triangle in $\widetilde{T}$ is mapped to a totally-geodesic ideal triangle in $\mathbb{H}^3$.  The edges of $\widetilde{T}$ are then the \textit{pleating/bending lines} of the pleated plane. Since  the pleated plane acquires an orientation from the orientation of $S$ (that is also acquired by the universal cover), we can uniquely define the \textit{bending angle} at each edge of $\widetilde{T}$ in the interval $[0,2\pi)$. 

The pleated plane is said to be $\rho$-equivariant, for a representation $\rho\mathbin{:}\Pi \to \pslc$, if  $\Psi(\gamma\cdot x) = \rho(\gamma) \cdot \Psi(x)$ for each $x\in \widetilde{S}$ and each $\gamma \in \Pi$. 

In this case $\rho$ also acquires a natural framing $\beta_\rho$ that maps each point  $p\in F_\infty$, which is a vertex of a triangle $\Delta$  in $\widetilde{T}$,  to the corresponding  ideal vertex of the ideal triangle $\Psi(\Delta)$.
Then we say that  the framed representation $\hat{\rho}= (\rho, \beta_\rho)$ is the \textit{monodromy} of the pleated plane.\vspace{.15mm}

In that case the  \textit{Fock-Goncharov coordinates} (or \textit{cross-ratio} coordinates)  for the representation $\hat{\rho}$, with respect to the triangulation $T$, are defined as follows:  for any lift $\tilde{e}$ of an edge $e \in T$, consider the two ideal triangles $\Delta_l$ and $\Delta_r$ that have edge $\tilde{e}$; the coordinate associated with that edge is then the cross-ratio of the four points determined by the $\beta_\rho$-images of the vertices of $\Delta_l$ and $\Delta_r$. 
More precisely, if these four vertices  of $T$ are $p_1,p_2,p_3$ and $p_4$ in counter-clockwise order on $\partial \mathbb{D}$, such that the edge $\tilde{e}$ is between $p_1$ and $p_3$, then the cross-ratio of $e$ is defined to be:
\begin{equation}\label{cross}
C(\hat{\rho}, e) = \frac{(z_1 - z_2) (z_3- z_4)}{(z_2- z_3) (z_1 - z_4)}
\end{equation} 
where $z_i = \beta_\rho(p_i)$ for $i=1,2,3,4$. 

The invariance of the cross-ratio under M\"{o}bius transformations, together with the $\rho$-equivariance of the above assignments then implies that each lift of $e$ is assigned the same complex number; hence the Fock-Goncharov coordinates are in fact associated with the edges of $T$ on $S$. 

Note that this definition requires that the four points corresponding to each $\tilde{e}$ are pairwise distinct, so this is only well-defined for a \textit{generic} framed representation. \vspace{.15mm}

\smallskip
We conclude the section with another result of Allegretti-Bridgeland that will be crucial in our proof.
In what follows a \textit{signed} triangulation is an ideal triangulation $T$, together with a choice of a sign ($\pm1$) at each puncture; we shall denote this $k$-tuple by $\epsilon$ (see \S9.3 of \cite{AllBrid}). 
Given a framed representation $\hat \rho$, such a signing $\epsilon$ determines a change of framing as follows: if a puncture is assigned a negative sign, and the monodromy around it is not parabolic or the identity map, then the new framing switches the choice of the fixed point of the monodromy around it. We shall denote this new framed representation by $\epsilon \cdot \hat\rho$. The Fock-Goncharov coordinates for $\hat\rho$  \textit{with respect to a signed triangulation} $(T,\epsilon)$ are then defined to be the cross-ratio coordinates (as above) for this new framed representation -- see Definition 9.5 in \cite{AllBrid}.

\begin{thm}[Theorem 9.1 of \cite{AllBrid}]\label{ab-thm91} Given a non-degenerate framed representation $\hat \rho$, there exists a choice of a signed triangulation $(T, \epsilon)$ of $S$, such that the Fock-Goncharov coordinates of $\hat{\rho}$ with respect to $(T,\epsilon)$ are well-defined and non-zero.
\end{thm}

\section{Proofs of Theorems \ref{thm1} -- \ref{thm4}}

\subsection{Step 1: framing the representation} 

We first  show: 

\begin{prop}\label{prop1} If $\rho \in   \chi(\Pi) $ is non-degenerate in the sense of Definition \ref{bel}, then for any choice of framing at the punctures the resulting framed representation $\widehat{\rho}$ is non-degenerate in the sense of Definition \ref{degen}. 

Conversely, given a non-degenerate framed representation $\widehat{\rho} \in \widehat{\chi}(\Pi)$, the corresponding representation $\rho \in   \chi(\Pi)$ obtained by forgetting the framing, is non-degenerate. 

\end{prop}
\begin{proof}

If $\hat{\rho}$ is a degenerate framed representation, then the underlying representation $\rho$ is degenerate:
this is because either condition (1) in Definition \ref{degen} is true, in which case (a) holds in Definition \ref{bel}, or condition (2) in Definition \ref{degen}  is true, in which case (b) holds  in Definition \ref{bel}. 
The contrapositive of this establishes the first assertion. 

We note that such a framing exists if  $\rho$ that is non-degenerate and has non-trivial monodromy around each puncture. Indeed,  consider the framing given by a choice of a fixed point in $\cp$ of the monodromy around each puncture, and then extending equivariantly to construct the $\rho$-equivariant map $\beta_\rho\mathbin{:}F_\infty \to \cp$.
This results in a framed representation $\hat{\rho}=(\rho, \beta_\rho)$ which cannot be a degenerate framed representation by the reasoning above.  In other words, $\hat{\rho}$ is a non-degenerate framed representation. Moreover, this is independent of the choice of the fixed point of the monodromy around each puncture  -- see (v) of Remarks 4.4 in \cite{AllBrid}.

For the second statement, let  $\rho$ be a degenerate representation, and let $\beta: F_\infty \to \cp$ be a $\rho$-equivariant map, i.e.\ a framing.  We shall show that the framed representation $\hat\rho = (\rho, \beta)$ is necessarily a degenerate framed representation. Let $\tilde{p} \in F_\infty$ be the lift of a puncture $p$, and let $\gamma$ be a loop around $p$. Note that the deck-translation on $\widetilde{S}$ corresponding to $\gamma$ fixes $\tilde{p}$; by the $\rho$-equivariance of $\beta$, this implies that $\beta(\tilde{p})$ is fixed by $\rho(\gamma)$.  In case $\rho$ satisfies (a) in Definition \ref{bel}, there is a point $p_0 \in \cp$ such that $\rho(\gamma)$ must be a parabolic or identity element fixing $p_0$; hence $\beta(\tilde{p}) = p_0$.  Hence the image of the map $\beta$ is thus a single point, and (1) in Definition \ref{degen} is satisfied.   In case $\rho$ satisfies (b) in Definition \ref{bel}, there is a pair of points $\{p_-,p_+\} \in \cp$ such that $\rho(\gamma)$ fixes one of them; by the same argument as above, we deduce $\beta(\tilde{p}) \in \{p_-,p_+\}$ for each $\tilde{p} \in F_\infty$. Therefore, (2) in Definition \ref{degen} is satisfied.  Thus, in both cases, the framed representation $\hat\rho$ is degenerate in the sense of Definition \ref{degen}, and we have proved the contrapositive of the second assertion.
\end{proof}

The following is then immediate:

\begin{cor}\label{thm1-part}
The image of the (unframed) monodromy map $\Phi \mathbin{:} \mathcal{P}^\circ_g({k}) \to  \chi(\Pi) $ is a subset of the set of non-degenerate representations with non-trivial monodromy around each puncture. 
\end{cor}

\begin{proof}
By Theorem \ref{ab-thm61} we know that the monodromy of any signed meromorphic projective structure in $\mathcal{P}^\ast_g(k)$ is a non-degenerate framed representation in $\widehat{\chi}(\Pi)$. By the second statement of the previous Proposition, the  representation in $\chi(\Pi)$ obtained by forgetting the framing, is non-degenerate. This is the image of the map $\Phi$ obtained by forgetting the signs at the punctures (see \eqref{commd}).  Since we have already assumed that the $\cp$-structures in $\mathcal{P}^\circ_g(k)$ have no apparent singularities, each puncture has non-trivial monodromy around it.\end{proof}

\subsection{Step 2: straightening the pleated plane}

Let $\hat\rho$ be the non-degenerate framed representation obtained in Step 1.

 By Theorem \ref{ab-thm91}, there is an an ideal triangulation $T$ and a signing $\epsilon$ that defines a new framing by switching the choice of fixed points of the monodromy around some of the punctures, such that the cross-ratio coordinates of the resulting framed representation $\epsilon \cdot \hat\rho$  with respect to $T$ are well-defined and non-zero.  In what follows, we shall continue to refer to this new framed representation by $\hat\rho$.  (Note that this is still a non-degenerate framed representation -- see Remarks 4.4 (v) of \cite{AllBrid}.)

Given this non-degenerate framed representation $\hat\rho$, and the ideal triangulation $T$ of the surface $S$ as above, we can build a pleated plane $\Psi:\widetilde{S} \to \HH^3$  as follows: for any triangle $\Delta$ in the lifted triangulation $\widetilde{T}$ in the universal cover $\widetilde{S}$, the vertices of $\Delta$ are three distinct elements $a,b,c\in F_\infty$. Since the Fock-Goncharov coordinate of any edge of $T$ is well-defined and non-zero, the image points $\beta_\rho(a), \beta_\rho(b), \beta_\rho(c)$ are three distinct points in $\cp$, and we define $\Psi$ to map $\Delta$ to the ideal triangle in $\HH^3$  with those three points as ideal vertices.  (Note that if we only know the conjugacy class of $\hat\rho$, i.e. as a point of the moduli space of framed representations, then $\Psi$ is well-defined up to post-composition by an isometry of $\HH^3$.) \vspace{.15mm}

We can define:

\begin{defn}[Straightening]\label{straight}  Recall that given a $\rho$-equivariant pleated plane $\Psi: \widetilde{S} \to \HH^3$ as above,  pleated along edges of a triangulation $\widetilde{T}$, each edge $\tilde{e}$ is assigned a complex cross-ratio $c(\tilde{e}) \in \C^\ast$.  We  define the \textit{straightened} plane  as a new map $\overline{\Psi}:\widetilde{S} \to \HH^3$ that is obtained by the assignment of \textit{real} cross-ratios $\{ \lvert c(\tilde{e})\rvert \text{ } \vert \text{ } e \in \widetilde{T} \}$ to each edge of the triangulation. Note that the image of the resulting pleated plane $\overline {\Psi}$ lies on the totally-geodesic equatorial plane $\mathbb{H}^2$.  If $\overline{\Psi}$ is equivariant with respect to the action of a Fuchsian group $\Gamma$ on  $\mathbb{H}^2$, we shall say that the straightened plane is the developing map of the hyperbolic surface $\mathbb{H}^2/\Gamma$. 
\end{defn}

In this section we shall prove:

\begin{prop}\label{prop2}  Let $\Psi$ be the pleated plane obtained from $\hat\rho$ as above. 
Then the  straightened plane $\overline{\Psi}$ is the developing map of a hyperbolic surface $\widehat{S}$ homeomorphic to $S_{g,k}$, such that each puncture corresponds to either a cusp or a geodesic boundary component.

Moreover, the pleating lines descend to a measured lamination $\lambda$ on $\widehat{S}$ that is a collection of weighted geodesic arcs, where the weight equals the ``bending angle" at the pleat. We include the geodesic boundary components as part of $\lambda$,  each with infinite weight (in this case there will be leaves of  $\lambda$ that spiral on to these boundary components). 
\end{prop}

\medskip 
 
The following lemma concerning a toy example, shall be useful in the proof of Proposition \ref{prop2}, since it will serve as a model for what happens in a neighborhood of a puncture. For this, we introduce the following terminology:\vspace{.15mm}

Let $\widetilde{A} \subset \mathbb{H}^2$ be the subset comprising infinitely many ideal triangles $\{ \Delta_i\}_{i\in \mathbb{Z}}$ enclosed by the collection of lines $\widetilde{T} = \{ \text{Re}(z) = m \text{ } \vert \text{ } m \in \mathbb{Z} \}$ together with the semi-circles orthogonal to $\partial \HH^2$ having endpoints on pairs of consecutive integers.  We shall consider the infinite cyclic group $\langle \tau_n \rangle$ acting on $\widetilde{A}$, where  $\tau_n$ is the translation $z\mapsto z + n$ for some integer $n\geq 1$.  The quotient $\widetilde{A}/\langle \tau_n\rangle$ will be referred to as an \textit{annular end}. Note that $\widetilde{T}$ descends to  an  ideal triangulation $T$ of the annular end,  comprising $n$ ideal triangles. 

A \textit{pleated annular end}  is an equivariant map $\psi\mathbin{:} \widetilde{A} \to \mathbb{H}^3$, such that the image of each triangle $\Delta_i$ is a totally-geodesic ideal triangle in $\HH^3$. Here, the map is equivariant with respect to the $\mathbb{Z}$-action on $\widetilde{A}$ generated  by the translation $\tau_n$ as above, and the action on $\HH^3$ by the cyclic group generated by a M\"{o}bius map $M\in \pslc$ which is said to be the \textit{monodromy} of the annular end.  

As in \S2.5, a pleated annular end as above can be recovered from the ordered tuple of  $n$ complex cross-ratios associated with the edges  of the ideal triangulation $T = \{  \text{Re}(z) = m \text{ } \vert \text{ } m=0,1,2\cdots, n-1\}$. We shall assume the pleated annular end is \textit{generic},  so that these $n$ complex cross-ratios are all non-zero. 

A straightening of a pleated annular end can then be defined exactly as in Definition \ref{straight}. The image of the resulting new map $\overline{\Psi}$ then lies on the equatorial plane, which is a totally geodesic copy of $\mathbb{H}^2$.  We shall assume that the induced orientation on the ideal boundary $\partial \mathbb{H}^2 = \mathbb{R}\cup \{\infty\}$ is the standard (increasing) ordering on $\mathbb{R}$. Moreover, the map $\overline{\Psi}$  is then equivariant with respect to the $\mathbb{Z}$-action on $\widetilde{A}$, as before, and an action on $\mathbb{H}^2$ (the equatorial plane) by the cyclic group generated by a \textit{real} M\"{o}bius transformation, i.e.\ an element $\overline{M} \in \pslr$.

The monodromy of the pleated annular end and its straightening can be determined as follows:

\begin{lem}\label{toy}  Consider a generic pleated annular end as above, and let $s_0,s_1,\ldots s_{n-1}$ be the complex cross-ratios associated with the $n\geq 1$ edges of the ideal triangulation $T$. Then, the monodromy $M\in \pslc$ of the annular end is 
\begin{enumerate}
\item loxodromic, if $\sum\limits_i \ln \lvert s_i \rvert  \neq 0$ 
\item parabolic or identity, if $\sum\limits_i \ln \lvert s_i \rvert = 0$ and  $\sum\limits_i Arg(s_i) \in 2\pi \mathbb{Z}  $, and 
\item elliptic, if $\sum\limits_i \ln \lvert s_i \rvert = 0$ but $\sum\limits_i Arg(s_i) \notin 2\pi \mathbb{Z} $. 
\end{enumerate} 
\end{lem}

\begin{proof}
It is best to work in the upper-half-space model of $\HHH^3$, such that the image of the ideal vertex at $\infty$ in $\widetilde{A}$ is $p_0=\infty$ in $\HH^3$.  See Figure 2.

 Let $\Delta_0,\Delta_1, \ldots \Delta_{n-1}$ be the totally geodesic ideal triangles in the $\Psi$-image of a fundamental domain of the $\mathbb{Z}$-action on $\widetilde{A}$. 
 
A computation using our definition of cross-ratio in \eqref{cross} shows that $Arg(s_i)$ equals the dihedral angle between $\Delta_i$ and $\Delta_{i+1}$, and $\lvert s_i \rvert$ is the (multiplicative) factor for the gluing along the geodesic line common to the two ideal triangles. (Equivalently, $\ln \lvert s_i \rvert$ is the shear parameter between the two triangles.) 

Suppose the vertices of $\Delta_i$ are $\infty, p_2$ and $p_3$  and the vertices of $\Delta_{i+1}$ are $\infty,p_3$ and $p_4$ such that the common geodesic side is between $\infty$ and $p_3$. Then  the isometry of $\HH^3$  that takes $\Delta_i$ to $\Delta_{i+1}$ mapping $p_2$ to $p_3$ and fixing $\infty$,  extends to an automorphism of the boundary $\C$  of the upper half-space,  and has the form $z\mapsto s_i z + b_i$ for some $b_i \in \C$.

The monodromy $M$ then extends the composition of such maps for $0 \leq i\leq n-1$, which is a map of the form $z \mapsto s_0s_1\cdots s_{n-1} z+ b$ for some $b\in \C$. 

The classification (1)-(3) is then immediate.  \end{proof}

\begin{figure}
  \centering
  \includegraphics[scale=0.45]{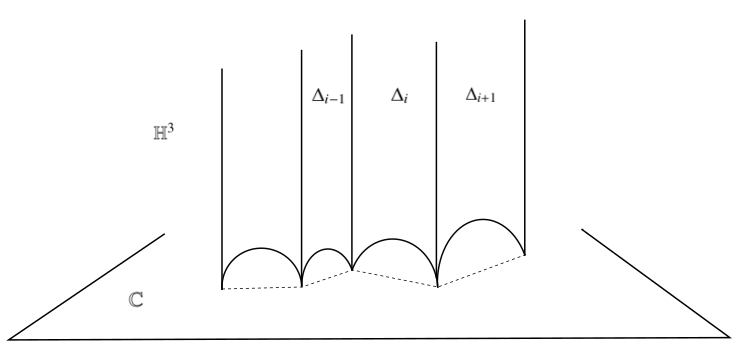}
  \caption{A pleated annular end in the upper half-space model of $\HH^3$. }
\end{figure}

\begin{cor}\label{cor-toy} Consider a pleated annular end comprising $n$ ideal triangles as above, determined by the cross-ratios $s_0,s_1,\ldots, s_{n-1}$. Then the straightened annular end, determined by the real cross-ratios $\lvert s_0\rvert ,\lvert s_1\rvert,\ldots, \lvert s_{n-1}\rvert$,  has monodromy $\overline{M}$ that is either 
\begin{enumerate}
\item[(i)] a parabolic isometry, in the  case $\sum\limits_i \ln \lvert s_i \rvert  = 0$, or 
\item[(ii)] a hyperbolic isometry of translation length $\sum\limits_i \ln \lvert s_i \rvert  \neq 0$. 
\end{enumerate} 
\end{cor}
\begin{proof}
Let $\Delta_0^\prime ,\Delta_1^\prime, \ldots \Delta_{n-1}^\prime$ be the totally geodesic ideal triangles in the $\overline{\Psi}$-image of a fundamental domain of the $\mathbb{Z}$-action on $\widetilde{A}$, and more generally  let $\Delta_i^\prime$ be the image of the ideal triangle $\Delta_i$ in  $\widetilde{A}$.  These are ideal triangles on $\mathbb{H}^2$ (the equatorial plane) with a common ideal vertex.  
Consider two such adjacent ideal triangles $\Delta_-$ and $\Delta_+$ that are images of successive ideal triangles in $\widetilde{A}$;  for ease of computation assume they have a common ideal vertex at $\infty$, a common geodesic line  between $\infty$ and $1$, and the  two other ideal vertices of $\Delta_-$ and $\Delta_+$ are $0$ and $\lambda \in \mathbb{R}$ respectively.  Then a calculation shows that the cross-ratio (as defined in \eqref{cross}) assigned to the common edge is $\lambda -1$. From this it follows that whenever the  real cross-ratio is positive, then the pair of ideal  triangles that is the image of two adjacent ideal triangles in $\widetilde{A}$, is embedded in $\mathbb{H}^2$ (i.e.\ they don't fold over). In our case, each real cross-ratio that determines the gluing of  $\Delta_i^\prime$ to $\Delta_{i+1}^\prime$  is positive for each $i$, so the entire image of $\widetilde{A}$ is  embedded in $\mathbb{H}^2$.  Let $\psi_i$ be the isometry of $\HH^2$  that takes $\Delta_i^\prime$ to $\Delta_{i+1}^\prime$  fixing the common ideal vertex $\infty$ and taking the other geodesic side of $\Delta_i^\prime$ containing $\infty$, to the common geodesic side between the two triangles, exactly as in the proof of the Lemma above. Then, in the upper half-plane model for $\mathbb{H}^2$, the isometry $\psi_i$  is  of the form $z\mapsto \lvert s_i \rvert z + c_i$ for some $c_i \in \mathbb{R}^+$.  (For instance, for the pair of ideal triangles $\Delta_\pm$ as above, the specified  isometry taking $\Delta_-$ to $\Delta_+$ is of the form $z\mapsto (\lambda -1) z +1$.) The monodromy $\overline{M}$ is then a composition of such maps across the triangles in the fundamental domain, which yields a map of the form $z \mapsto \lvert s_0s_1\cdots s_{n-1} \rvert z+ c$ for some $c\in \mathbb{R}^+$.
\end{proof}

\textit{Remark.} In case of (i) above, the straightened annular end is in fact the developing map of a cusp. In the case (ii), it is the developing map of an incomplete end whose completion has a geodesic boundary of length $\sum\limits_i \ln \lvert s_i \rvert  \neq 0$. (See, for example, Chapter 7.4 of \cite{Martelli}.) \\

\smallskip 

\begin{proof}[Proof of Proposition \ref{prop2}]

Let $\overline{\Psi}\mathbin{:} \widetilde{S} \to \HH^3$ be the straightened pleated plane. As a consequence of Definition \ref{straight}, $\overline{\Psi}$ is the developing map of a hyperbolic structure on $S_{g,k}$ obtained by gluing the ideal triangles of the triangulation $T$, where the real cross-ratios associated with the edges of $T$ determine the gluing. We call this hyperbolic structure $\widehat{S}$.

For the $i$-th puncture, consider the edges $e_0,e_1,\ldots, e_{n-1}$ of $T$ incident on it, and consider a lift of each of these edges in $\widetilde{S}$ that are all incident to  the same ideal vertex in $F_\infty$.  By Corollary \ref{cor-toy}, it follows that 
$i$-th puncture corresponds to either a cusp (if (i) holds, that is, the sum of the real shear parameters for $e_i$ for $0\leq i\leq n-1$ is zero) or a geodesic boundary component (if (ii) holds).

To each edge $e$ of the ideal triangulation $T$ on $\widehat{S}$,  we assign a weight $\alpha(e) \in [0, 2\pi)$ that is the argument of the complex cross-ratio associated with edge. (In fact $\alpha(e)$ also equals the dihedral angle at $\Psi(e)$ between the two totally-geodesic ideal triangles in $\HH^3$ that are the $\Psi$-images of the adjacent ideal triangles.)  When case (ii) of Corollary \ref{cor-toy} holds above,  the corresponding  geodesic lines of  the lifted ideal triangulation $\tilde{T}$ accumulate on to the lift of the geodesic boundary component of $\widehat{S}$. In the quotient surface $\widehat{S}$, these descend to finitely many leaves that spiral on to that boundary geodesic.  The subset of  $T$ comprising geodesic lines equipped with positive weights as defined above, together with the closed geodesic boundary components (each with infinite weight), form a measured lamination on $\widehat{S}$. 
\end{proof}

\subsection{Step 3: obtaining the meromorphic projective structure}

Let $\widehat{S}$ and $\lambda$ be the hyperbolic surface and measured geodesic lamination, respectively, as obtained from Proposition \ref{prop2} in Step 2. 

We shall now prove:

\begin{prop}\label{prop3} The $\cp$-structure obtained by grafting $\widehat{S}$ along $\lambda$ is a meromorphic projective structure in  $\mathcal{P}_g(k)$.
\end{prop}

\smallskip

\textit{Remark.} In the next section we shall see that this $\cp$-structure is what we wanted, i.e.\ it has monodromy $\rho$ (and hence in fact lies in $\mathcal{P}_g^\circ(k)$).  In general, the projective structure obtained after grafting some cusped  hyperbolic surface along a disjoint  collection of weighted simple geodesics could have apparent singularities, i.e.\ punctures having trivial monodromy. We leave a grafting description of \textit{all} $\cp$-structures in $\mathcal{P}_g(k)$ to a future paper.  \vspace{.15mm} 
 
\medskip

We first define:

\begin{defn}[Semi-infinite lune]\label{sil}
A \textit{semi-infinite lune} $L_\infty$ is the surface obtained by a construction as in  Definition \ref{imlune}, in the case when the angle $\alpha =\infty$. This can be described as follows: consider $\cp$ with a slit along the positive imaginary axis $\ell$ from $0$ to $\infty$, and let $\ell^+_0$ and $\ell^-_0$ be the resulting sides. Take a collection of such copies of $\cp$ with an identical slit, indexed by $\mathbb{N}$, and identify $\ell^-_i$, the right side of the slit on the $i$-th copy, with $\ell^+_{i+1}$, the left side of the slit on the $(i+1)$-th copy, for each $i\in \mathbb{N}$.

(Note that  just like  $L_\alpha$ for $\alpha>2\pi$,  $L_\infty$  admits an immersion to $\cp$ since there is a natural projection of each copy of $\cp$ to a fixed $\cp$.) 
\end{defn}

We observe:

\begin{lem}\label{schwarz1} The semi-infinite lune $L_\infty$ can be identified with the half-plane $\mathbb{H}$ via a conformal homeomorphism. This identification can be chosen such that the immersion  of $L_\infty$ to $\cp$ is $f_{a}\mathbin{:}\mathbb{H} \to \cp$ where $f_a(z) = ie^{a z}$, for any $a \in \mathbb{R}^+$.

\end{lem}
\begin{proof} 
Fix an $a \in \mathbb{R}^+$.
The $n$-th copy of $\cp \setminus \ell$  as in Definition \ref{sil} can be identified with a horizontal strip $H_n= \{ \frac{2\pi n}{a}  < \text{Im}(z) < \frac{2\pi (n+1)}{a} \} \subset \C$ via the  biholomorphism $\phi\mathbin{:}H_n \to \cp \setminus \ell$ given by the exponential map $f_a(z) = i e^{a z}$. (The multiplicative factor of $i$ because $\ell$ is the positive \textit{imaginary} axis.) See Figure 3.

Since $f_a (z+ i2\pi/a) = f_a(z)$, and $L_\infty$ comprises copies of $\cp\setminus \ell$ indexed by $n\geq0$ with identifications along the sides of the slit along $\ell$ in the each copy, we obtain  a global  biholomorphism of   $\mathbb{H} = \bigcup\limits_{n\geq 0} H_n $ with $L_\infty$. 
By our construction, the immersion from $L_\infty$ to $\cp$ is then precisely $f_a$.
\end{proof}

\begin{figure}
  \centering
  \vspace{.1in}
  \includegraphics[scale=0.36]{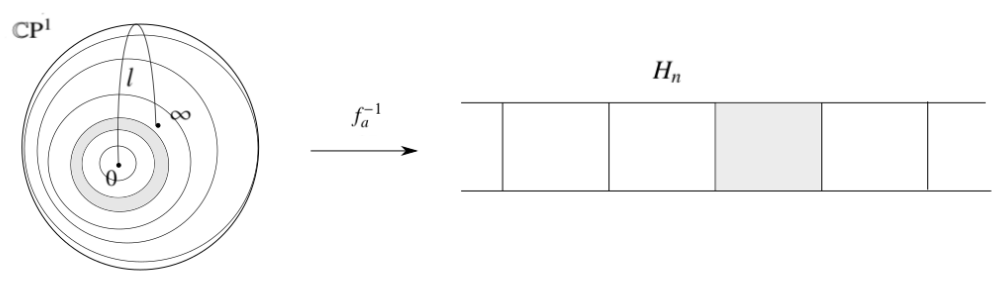}
  \caption{The  conformal identification of $\cp \setminus \ell$ with an infinite strip. }
\end{figure}

\begin{proof}[Proof of Proposition \ref{prop3}] 

We have already seen in \S2.2 that grafting a hyperbolic surface along a measured lamination (in this case, a collection of disjoint weighted geodesic lines) can be defined by inserting lunes in the developing image of the universal cover. This results in a $\cp$-structure on the underlying topological surface (in this case, $S_{g,k}$). Thus, it only remains to verify that the Schwarzian derivative of the developing map has a pole of order at most two, at each puncture, when computed with respect to a reference $\cp$-structure on the closed surface $S_g$, in which a neighborhood of the puncture is the punctured disk $\mathbb{D}^\ast$. 

\smallskip

By Proposition \ref{prop2}, for each puncture of $S_{g,k}$, there are two possibilities (see Figure 4):\\

\textit{Case 1: $\widehat{S}$ has a corresponding geodesic boundary component $c$.}

In this case $\lambda$ will include leaves that spiral onto the boundary component $c$, where  $c \subset \lambda$, and is equipped with infinite weight. 

The developing image of the universal cover of $\widehat{S}$ is a subset of the round disk $\mathbb{H}^2$. A lift of $c$ corresponds to a geodesic line that cuts off a half-plane in $\mathbb{H}^2$. We can assume that this lift $\ell$ is the positive imaginary axis, and the developing image of the surface lies to its right.

Since $c$ has infinite weight, from our description in \S2.2,  grafting along it results in attaching a  semi-infinite lune $L_\infty$  by identifying $\partial_- L_\infty$ with $\ell$. 
Moreover, the line $\ell = \partial_- L_\infty$ is invariant under the hyperbolic translation $\gamma(z) = a z$ which is the monodromy around the boundary component $c$. The  action of the infinite cyclic group $\langle \gamma \rangle$  extends to the entire semi-infinite lune $L_\infty$ in an obvious way: on each copy of $\cp$, $\gamma$ acts with the same expression as before, where $z$ is now the coordinate on that copy. At the level of the quotient surface, this ``infinite grafting" results in attaching a semi-infinite Euclidean cylinder (the quotient $\mathbb{H}/\langle z \mapsto z+ 1\rangle $  of the upper half plane -- see Lemma \ref{schwarz1}), to the boundary component $c$ of $\widehat{S}$. 

In the proof of Lemma \ref{schwarz1}, we can choose the conformal identification  $\mathbb{H} \cong L_\infty$ such that the immersion is $f_a\mathbin{:}L_\infty \to \cp$, where $a$ is the dilation factor of $\gamma$ as above. This is equivariant with respect to  the $\mathbb{Z}$-action  generated by the translation $z\mapsto z+1$ on $\HH$, and the $\mathbb{Z}$-action generated by $\gamma$ on $L_\infty$.  That is, we have $f_a(z+1) = e^af_a(z)$ for any $z\in \HH$.

This immersion $f_a$ is the developing map of the $\cp$-structure obtained after grafting, when restricted to the semi-infinite lune $L_\infty$. From \eqref{schwarz} we can easily compute that the Schwarzian derivative $S(f_a) =  - \frac{1}{2}a^2dz^2$ on $\mathbb{H}$.  This descends to a quadratic differential with a pole of order $2$ on  $\mathbb{H}/ \langle z\mapsto z+1 \rangle$, which can be conformally identified with the punctured disk $\mathbb{D}^\ast$ on $\widehat{S}$ via the map  $z\mapsto w:= e^{2\pi iz}$. Indeed,  in the new coordinate $w$, the constant quadratic differential $dz^2$ has the expression $- \frac{1}{4\pi^2} w^{-2} dw^2$. That is, the Schwarzian derivative $S(f_a)$ above, descends to a holomorphic quadratic differential on $\mathbb{D}^\ast$  with a pole of order two at the puncture.

\smallskip

\begin{figure}
  \centering
  \vspace{.1in}
  \includegraphics[scale=0.28]{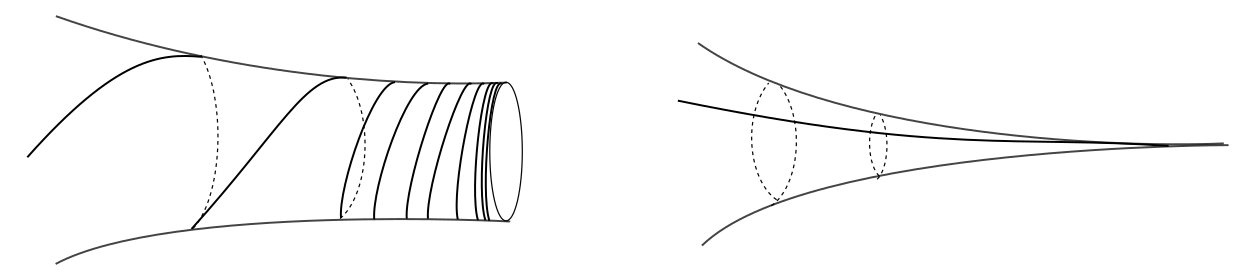}
  \caption{The two possibilities of a  neighborhood of a puncture on $\widehat{S}$: an incomplete end that is completed by adding a geodesic boundary (left), and a complete cusp (right). In the former case, the boundary component has at least one leaf of the lamination $\lambda$ spiralling onto it, as shown. }
\end{figure}

\textit{Case 2: $\widehat{S}$ has a corresponding cusp boundary.}

In this case $\lambda$ could have finitely many leaves exiting the cusp. Let $U$ be a neighborhood of the puncture after grafting;  once again, we conformally identify $U$ with the punctured disk $\mathbb{D}^\ast$. Thus, the universal cover  of $U$ can be identified with the upper half-plane  $\mathbb{H}$, so that the universal covering map is $z\mapsto w:= e^{2\pi iz}$ as before. The developing map of the $\cp$-structure  (after grafting)  restricted to $\mathbb{H}$  is then an equivariant conformal immersion  $f\mathbin{:}\mathbb{H} \to \cp$, where the action on $\mathbb{H}$ is generated by the translation $z\mapsto z+1$, and the action on $\cp$ is generated by the monodromy around the puncture.  In what follows, the fundamental domain of the $\mathbb{Z}$-action on $\mathbb{H}$ is an infinite strip $\Delta$.

There are two sub-cases:

(i) In case $\lambda$ has \textit{no} leaf exiting the cusp boundary, then the monodromy around the puncture remains the same parabolic element of $\pslr$ as in the ungrafted hyperbolic surface, and the developing map $f$ on $\mathbb{H}$ (and hence $\Delta$)  is just the identity map $f(z) =z$.  On the punctured disk $U$, the developing map thus has the expression $\bar{f}(w) = \frac{1}{2\pi i} \ln w$.

(ii) In case $\lambda$ has leaves of total weight $\alpha>0$ exiting the cusp boundary, then $f$ will map a neighbourhood of $\infty$ in $\Delta$, to a neighborhood of $\infty$ in a lune $L_\alpha$, such that the two vertical geodesics bounding $\Delta$ map to the circular arcs bounding $L_\alpha$.  (See Figure 5.)  Such a conformal map will have the asymptotic expression $f(z) = e^{- i \alpha z}$ for $\lvert z \rvert \gg 1$, and hence on the punctured disk, has the expression $\bar{f}(w) = w^{-\alpha/2\pi}$. 

In both possibilities (i) and (ii), a simple computation using \eqref{schwarz-deriv} then yields that the Schwarzian derivative $S(\bar{f})$ has a pole of order  two at the puncture in $U \cong \mathbb{D}^\ast = \{ \lvert w \rvert <1\}$. 
\end{proof}

\begin{figure}
  \centering
  \vspace{.1in}
  \includegraphics[scale=0.4]{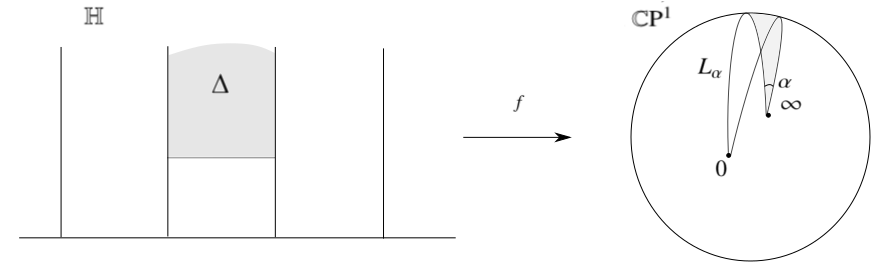}
  \caption{The developing map $f\mathbin{:}\Delta \to \cp$ in case 2(ii) of the proof of Proposition \ref{prop3}. }
\end{figure}

\subsection{Completing the proofs} 

We can now complete the proofs of Theorems \ref{thm1}, \ref{thm2}, \ref{thm3} and \ref{thm4} stated in \S1.

\begin{proof}[Proof of Theorem \ref{thm2}]

We already know from Theorem \ref{ab-thm61} that the image of $\widehat{\Phi}$ is contained in the subset of non-degenerate framed representations. Moreover, from the assumption of no apparent singularities, any such representation will have non-trivial monodromy around each puncture. Thus, it remains to show surjectivity.

Consider a non-degenerate framed representation $\hat{\rho}^\prime  \in  \widehat{\chi}(\Pi)^\ast$ with non-trivial monodromy around each puncture. By Theorem \ref{ab-thm91}, there is a signed ideal triangulation $(T,\epsilon)$ such that $\hat\rho :=\epsilon \cdot \hat{\rho}^\prime$ has well-defined non-zero Fock-Goncharov coordinates with respect to $T$.  (Recall that  $\hat{\rho}$ has a new framing obtained by switching the choice of fixed point at the punctures with negative signs.) By the construction in \S2.6 we obtain a pleated plane $\Psi: \widetilde{S}  \to \HH^3$ that is a $\rho$-equivariant immersion that maps each triangle in $T$ to a totally-geodesic ideal triangle in $\HH^3$. 

By Proposition \ref{prop2} (Step 2) we know that the \textit{straightening} of the pleated plane results in a new embedding $\overline{\Psi}:\widetilde{S} \to \HH^3$ whose image lies on the equatorial plane (which can be identified with the Poincar\'{e} disk model of $\mathbb{H}^2$), that is equivariant with respect to a Fuchsian representation. The quotient is a hyperbolic surface $\widehat{S}$ homeomorphic to $S_{g,k}$, such that the punctures are either cusps or can be completed by adding a geodesic boundary. Moreover, the pleating lines determine a measured lamination $\lambda$ on $\widehat{S}$. By Proposition \ref{prop3} (Step 3) the grafting of $\widehat{S}$ along $\lambda$ results in a meromorphic projective structure $P\in \mathcal{P}_g(k)$. 

Recall that each leaf of $\lambda$ is an edge $e$ of the triangulation $T$; since the weight $\alpha(e)$ on the leaf equals the argument of the complex cross-ratio $s(e)$ by definition (see the proof of Proposition \ref{prop2}), the pleated plane defined by the grafting (see \S2.2) is precisely $\Psi$.   This pleated plane is $\rho$-equivariant, and hence  the monodromy of the $\cp$-structure $P$ is $\rho$.  Moreover, by construction the monodromy of the pleated plane $\Psi$ is exactly $\hat \rho$.  

Indeed, the framing determined by $\Psi$ depends only on the leaves of $\lambda$ of \textit{finite} weight  (\textit{c.f.} \S3.2.3 of \cite{GM2}). To see this,  let $\gamma$ be a geodesic boundary component of $\widehat{S}$; this is a leaf of $\lambda$ of infinite weight  (see the proof of  Proposition \ref{prop2}).  Also, let $F$ be a fundamental domain of the action of $\Pi$ on the universal cover $\widetilde{S}$. Then the framing determined by $\Psi$ assigns to the corresponding puncture a choice of one of the endpoints of the lift $\widetilde{\gamma}$ in  $F$.   Infinite grafting along $\gamma$  inserts  a semi-infinite lune at $\widetilde{\gamma}$, and this does not change these endpoints. 

 At each puncture, we record the sign of the exponent that corresponds to the choice of the eigenline of the monodromy around it, as described above; this is where we use our assumption that this monodromy is non-trivial (\textit{c.f.} Lemma \ref{pmon} and the remark following it). We then obtain  a \textit{signed} meromorphic projective structure $P^\ast \in \mathcal{P}^\ast_g(k)$ with \textit{framed} monodromy $\hat \rho$.

Finally, we change the signing of $P^\ast$ according to $\epsilon$, that is,  switch the signs at the punctures that $\epsilon$ assigns a negative sign.  If $P^{\ast\ast}$ is the resulting signed $\cp$-structure, then the framed monodromy of $P^{\ast\ast}$ is $\epsilon \cdot \hat \rho = \epsilon \cdot \epsilon \cdot \hat\rho^\prime = \hat\rho^\prime$, which is the \textit{original} framed representation. This completes the proof.  \end{proof}

\smallskip

\begin{proof}[Proof of Theorem \ref{thm1}]

By Corollary \ref{thm1-part}, it again suffices to prove the surjectivity of the map $\Phi$.
Let $\rho \in \large\chi(\Pi)$ be a non-degenerate representation in the sense of Definition \ref{bel} such that no puncture has trivial monodromy around it. By Proposition \ref{prop1} (Step 1), there is a framing $\beta$ such that $\hat{\rho} = (\rho,\beta)$ is a non-degenerate framed representation. 
By Theorem \ref{thm2}, there exists a (signed) meromorphic projective structure $P^{\ast} \in \mathcal{P}^\ast_g(k)$ with framed monodromy $\hat{\rho}$. 
 Forgetting the signing of the meromorphic projective structure, and the framing of the representation via the maps $\pi_0$ and $\pi$ respectively (see \eqref{commd}),  we conclude that the image of the (unsigned) $\cp$-structure $\pi_0(P^{\ast}) \in \mathcal{P}_g^\circ(k)$ under the monodromy map $\Phi$ is $\rho$.
 This completes the proof. 
\end{proof}

\smallskip

\begin{proof}[Proof of Theorem \ref{thm3}]

Let $\rho\in \large\chi(\Pi)$ lie in the image of the monodromy map $\Phi$, as described in Theorem \ref{thm1}. As described in the proof, there exists a $\cp$-structure $P \in \mathcal{P}^\circ_g(k)$ with monodromy $\rho$, that is obtained by grafting a hyperbolic surface $\widehat{S}$ homeomorphic with cusps and/or geodesic boundaries at the punctures, along a measured lamination $\lambda$. The lamination $\lambda$ comprises a collection of disjoint weighted geodesic lines, that are in fact edges of an ideal triangulation $T$ on $\widehat{S}$. (In the case $\widehat{S}$ has geodesic boundary components, some leaves of $\lambda$  necessarily spiral on to them.) 

 Let $l$ be a leaf of $\lambda$ with weight $\alpha>0$, or an edge of the triangulation $T$ of weight $\alpha=0$, that we denote,  as usual, by $\alpha \cdot l$.  For each $n \in \mathbb{N}$,  consider the measured lamination 
 \begin{center}
  $\lambda_n = \left.
  \begin{cases}
    \left(\lambda \setminus \{\alpha \cdot  l\}\right)   \cup \{ (\alpha + 2\pi n)\cdot l\} & \text{when } \alpha>0 \\
    \lambda  \cup \{ 2\pi n\cdot l\} & \text{when } \alpha=0 \\
  \end{cases}\right.$
  \end{center}
Then the  projective structure $P_n \in \mathcal{P}^\circ_g(k)$ obtained by grafting $\widehat{S}$ along $\lambda_n$ differs from $P$ by a $2\pi$-grafting, and hence has the same monodromy $\rho$, for each $n\geq 1$. 
 \end{proof}

\smallskip

\begin{proof}[Proof of Theorem \ref{thm4}]
Let $\mathbb{P}$ be the set of punctures such that the corresponding entry of $\mathfrak{n}$ is at most two; these are the \textit{internal marked points} in the terminology of \cite{AllBrid}. We shall denote the rest of the punctures by $\mathbb{P}^\prime$.  Our proof, that we shall now describe, is a straightforward combination of techniques of \cite{GM2} with those in the preceding sections.

As described in \S2.1 and \S2.2 of \cite{GM2} (see also \cite{AllBrid}), the $\cp$-structures in $\mathcal{P}_g^\ast(\mathfrak{n})$ can be thought of having an underlying \textit{marked and bordered surface} $S_g(\mathfrak{n})$. This can be equipped with a hyperbolic metric such that the points in $\mathbb{P}$ correspond to cusps, and the punctures in $\mathbb{P}^\prime$ correspond to geodesic boundary components with distinguished points on it (the number of which is two less than the corresponding entry of  $\mathfrak{n}$).   In what follows, \textit{boundary arcs} shall refer to the arcs between these distinguished points; also, by passing to the universal cover $\widetilde{S}$, the lifts of the punctures in $\mathbb{P}$ and distinguished points determine a \textit{Farey set} $F_\infty$  that generalizes that introduced in \S2.5.

Let $\hat \rho = (\rho,\beta) \in \widehat{\chi}_g(\mathfrak{n})$ be a  framed representation, that is non-degenerate  in the sense of Definition 4.3 of \cite{AllBrid}. Equivalently, in our terminology, $\hat\rho$ does not satisfy conditions (a) and (b) of Definition \ref{degen}, where $\mathbb{P}$ is the set of  punctures taken into consideration for (b), and neither does the following additional condition hold:
\begin{center}
(3) There is some boundary arc $I$ and some lift $\tilde{I}$ in the universal cover whose  endpoints in $F_\infty$ are assigned the same point in $\cp$ by  the framing $\beta$.
\end{center}
(See $(D1)$ in \S2.4.2 of \cite{GM2}, and (D1) of Definition 4.3 in \cite{AllBrid}.) 

Moreover, we shall assume that $\hat\rho$ has non-trivial monodromy around each puncture in $\mathbb{P}$. 

Recall from \S2.5 of  \cite{GM2} that an ideal triangulation $T$ of $S_g(\mathfrak{n})$ has vertices at $\mathbb{P}$ and at the distinguished points on each geodesic boundary component;  moreover, a  \textit{signing} is a map $\epsilon\mathbin{:} \mathbb{P} \to \{-1,1\}$ (see \S9.3 of \cite{AllBrid}).
By Theorem 9.1 of \cite{AllBrid} there is a signed triangulation $(T, \epsilon)$ such that the Fock-Goncharov coordinates (or cross-ratio coordinates)  of $\hat\rho$ with respect to $(T,\epsilon)$ are well-defined and non-zero. In other words, these coordinates are well-defined for the framed representation $\epsilon\cdot \hat \rho$ obtained by switching the choice of fixed point of each monodromy around a puncture in $\mathbb{P}$, exactly as in \S2.6. 
Then, following the construction in \S3.2, we  obtain a pleated plane $\Psi\mathbin{:} \widetilde{S} \to \HH^3$ with pleating lines that are lifts of the edges of the ideal triangulation $T$. This pleated plane $\Psi$ has monodromy  $\epsilon \cdot \hat \rho$.

A \textit{straightening} $\overline{\Psi}$ of this pleated plane is defined exactly as in Definition \ref{straight};  this is the developing map of a hyperbolic surface $\widehat{S}$ homeomorphic to $S_{g,k}$. From the proof of Proposition \ref{prop2}, the punctures in $\mathbb{P}$ correspond to cusps or geodesic boundary components of $\widehat{S}$, and from the description in \S3.2.1 of \cite{GM2}, the punctures in $\mathbb{P}^\prime$ correspond to ``crown ends".  The pleating lines descend to this hyperbolic surface $\widehat{S}$ to define a measured geodesic lamination $\lambda$,  which comprises 
\begin{itemize}
\item a collection of disjoint isolated geodesic lines with finite weight or transverse measure (equal to the argument of the complex cross-ratio coordinate), and
\item the geodesic boundary components and geodesic sides of the crown ends, each with infinite weight. 
\end{itemize}
We then obtain a meromorphic projective structure $P^\ast \in \mathcal{P}_g^\ast(\mathfrak{n})$ by grafting $\widehat{S}$ along $\lambda$. Here the signing is determined as follows: the framing identifies a choice of fixed point for the (non-trivial) monodromy around each puncture in $\mathbb{P}$, and each such choice corresponds to a sign by Lemma \ref{pmon}.  Since the pleated plane defined by this grafting recovers $\Psi$, the framed monodromy of $P^\ast$ is $\epsilon \cdot \hat \rho$. 

By Proposition \ref{prop3}, the developing map will have poles of order at most two at the punctures in $\mathbb{P}$, and by Proposition 4.2 of \cite{GM1}, it has poles of higher order at the punctures in $\mathbb{P}^\prime$.  (Note that these results rely only on a \textit{local} computation of the Schwarzian derivative around each puncture.)

As in the conclusion of the proof of Theorem \ref{thm2}, we then switch the signs at $\mathbb{P}$ according to $\epsilon$, to obtain the desired signed meromorphic structure $P^{\ast\ast}  \in \mathcal{P}_g^\ast(\mathfrak{n})$ that has framed monodromy $\hat\rho$. This proves that $F$ surjects on to the subset of  $\widehat{\chi}_g(\mathfrak{n})$ we specified.

On the other hand,  by Theorem 6.1 of \cite{AllBrid}, the monodromy of any meromorphic projective structure in $\mathcal{P}_g^\ast(\mathfrak{n})$ is necessarily a non-degenerate framed representation, which by our assumption of no apparent singularities, has non-trivial monodromy around each puncture. This completes the proof.   \end{proof}

\bibliographystyle{amsalpha}
\bibliography{mrefs}

\end{document}